\documentclass{amsart}

\usepackage{amsthm,amssymb}

\newtheorem{theorem}{Theorem}[section]
\newtheorem{lemma}[theorem]{Lemma}
\newtheorem{proposition}[theorem]{Proposition}
\newtheorem{definition}[theorem]{Definition}
\newtheorem{claim}[theorem]{Claim}
\newtheorem{corollary}[theorem]{Corollary}
\newtheorem{remark}[theorem]{Remark}
\newtheorem{fact}[theorem]{Fact}

\renewcommand{\P}{\mathbb{P}}
\newcommand{\R}{\mathbb{R}}

\newcommand{\A}{\mathbb{A}}
\newcommand{\C}{\mathbb{C}}
\newcommand{\Q}{\mathbb{Q}}
\newcommand{\M}{\mathbb{M}}
\newcommand{\X}{\mathbb{X}}
\newcommand{\Y}{\mathbb{Y}}

\newcommand{\Coll}{\mathop{\mathrm{Coll}}}
\newcommand{\dom}{\mathop{\mathrm{dom}}}

\newcommand{\cf}{\mathop\mathrm{cf}}
\newcommand{\ran}{\mathop{\mathrm{ran}}}
\newcommand{\otp}{\mathop{\mathrm{otp}}}
\newcommand{\Add}{\mathop{\mathrm{Add}}}
\newcommand{\restrict}{\upharpoonright}
\newcommand{\Ult}{\mathop{\mathrm{Ult}}}
\newcommand{\crit}{\mathop{\mathrm{crit}}}
\newcommand{\stem}{\mathop{\mathrm{stem}}}

\renewcommand{\top}{\mathop{\mathrm{top}}}

\title{Successive failures of approachability}
\author{Spencer Unger}
\date{February 14, 2017}

\begin{document}

\begin{abstract}  Motivated by showing that in ZFC we cannot construct a special
Aronszajn tree on some cardinal greater than $\aleph_1$, we produce a model in
which the approachability property fails (hence there are no special Aronszajn
trees) at all regular cardinals in the interval $[\aleph_2,
\aleph_{\omega^2+3}]$ and $\aleph_{\omega^2}$ is strong limit. \end{abstract}

\maketitle

\section{Introduction}

In the 1920's, K\"onig \cite{konig} proved that every tree of height $\omega$
with finite levels has a cofinal branch. In the 1930's Aronszajn \cite{kurepa}
showed that the analogous theorem for $\omega_1$ fails.  In particular he
constructed a tree of height $\omega_1$ whose levels are countable which has no
cofinal branch.  Such trees have come to be known as Aronszajn trees.  The first
Aronszajn tree is \emph{special} in the sense that there is a function $f:T \to
\omega$ such that $f(s) \neq f(t)$ whenever $s$ is below $t$ in $T$.  This
function $f$ witnesses that $T$ has no cofinal branch.

These two theorems provide a strong contrast between the combinatorial
properties of $\omega$ and $\omega_1$.  K\"onig's theorem shows that $\omega$
has a certain compactness property, while an Aronszajn tree is a canonical
example of a noncompact object of size $\omega_1$.  These properties admit
straightforward generalizations to higher cardinals.  We say that a regular
cardinal $\lambda$ has the \emph{tree property} if it satisfies a higher analog
of K\"onig's theorem.  If $\lambda$ does not have the tree property, then we
call a counter example a $\lambda$-Aronszajn tree.  The natural question is:
Which cardinals carry Aronszajn trees?

A full answer to this question is connected to the phenomena of independence in
set theory and large cardinals.  The first evidence of this
comes from a theorem of Specker \cite{specker} which shows that Aronszajn's
construction can be generalized in the context of an instance of the generalized
continuum hypothesis.  In particular if $\kappa^{<\kappa} = \kappa$, then there
is a special $\kappa^+$-Aronszajn tree (for the appropriate generalization of
the notion of special).  On the other hand, the tree property has a strong
connection with the existence of large cardinals.  We say that an uncountable
cardinal is \emph{weakly compact} if it satisfies a higher analog of infinite
Ramsey's theorem.  By theorems of Tarski and Erd\"os \cite{tarskierdos} and Monk
and Scott \cite{monkscott}, an uncountable cardinal $\lambda$ is weakly compact
if and only if it is inaccessible and has the tree property.

The invention of forcing provided method for proving the consistency of the tree
property at accessible cardinals.  An early theorem of Mitchell and Silver
\cite{mitchell}, shows that the tree property at $\omega_2$ is equiconsistent
with the existence of a weakly compact cardinal.  So if the existence of a
weakly compact cardinal is consistent, then it is impossible to construct an
$\omega_2$-Aronszajn tree from the usual axioms of set theory.  Moreover, the
assumption that a weakly compact cardinal is consistent is necessary.  This
result gives an approach to resolving which cardinals have Aronszajn trees.  The
conjecture is that if the existence of enough large cardinals is consistent,
then we cannot prove the statement that for some $\lambda$ there is
$\lambda$-Aronszajn tree.  A weaker goal which captures many of the difficulties
of this conjecture is to show that in ZFC one cannot prove that there is a
cardinal $\lambda$ which carries a \emph{special} $\lambda$-Aronszajn tree.

It is this weaker goal that we address in this paper.  We prove

\begin{theorem} Under suitable large cardinal hypotheses it is consistent that
$\aleph_{\omega^2}$ is strong limit and the approachability property fails for
every regular cardinal in the interval $[\aleph_2, \aleph_{\omega^2+3}]$.
\end{theorem}

As we will mention below, the failure of the approachability property at a
cardinal $\lambda$ is a strengthening of the nonexistence of special
$\lambda$-Aronszajn trees.  So the theorem represents partial progress towards
the construction of a model with no special Aronszajn trees on any regular
cardinal greater than $\aleph_1$ and hence the weaker goal above.

Our theorem combines two approaches to the problem, which have until now seemed
incompatible.  The first is a ground up approach where one constructs models
where a longer and longer initial segments of the regular cardinals carry no
special trees (or even have the tree property).  The major advances in this
approach are due to Abraham \cite{abraham}, Cummings and Foreman
\cite{cummingsforeman}, Magidor and Shelah \cite{ms}, Sinapova
\cite{sinapova1}, Neeman \cite{neeman} and the author \cite{ungernsl} for the
tree property, and Mitchell \cite{mitchell} and the author \cite{ungerind2} for
the nonexistence of special trees.  This approach cannot continue through the
first strong limit cardinal without some changes suggested by the second
approach.

The second is an approach for dealing with the successors of singular
strong limit cardinals.  By Specker's theorem if $\nu$ is singular strong
limit and there are no special $\nu^{++}$-Aronszajn trees, then $2^\nu > \nu^+$.
So the \emph{singular cardinals hypothesis} fails at $\nu$.  The singular
cardinals hypothesis is an important property in the study of the continuum
function on singular cardinals and obtaining a model where it fails requires the
existence of large cardinals.  Any model for the non existence of special trees
above $\aleph_1$ must be a model where GCH fails everywhere.  Such a model was
first constructed by Foreman and Woodin \cite{FW} using a complex Radin forcing
construction.

For some time, a major problem for the second approach was whether it is
consistent to have the failure of SCH at $\nu$ and the nonexistence of special
Aronszajn trees at $\nu^+$.  This was resolved by Gitik and Sharon
\cite{gitiksharon} and their result was later improved by Neeman
\cite{neemansch} to give the tree property.  Note that by our remarks above such
models are required to get the nonexistence of special trees (or the tree
property) at $\nu^+$ and $\nu^{++}$ where $\nu$ is singular strong limit.
Further advances in this area are due to Cummings and Foreman
\cite{cummingsforeman}, the author \cite{ungersucc}, Sinapova
\cite{sinapova1,sinapova2,sinapova3} and the author and Sinapova \cite{su4}.

The main forcing in this paper combines the two approaches outlined above.  In
particular it combines the ground up approach in \cite{ungerind2} with a version
of Gitik and Sharon's \cite{gitiksharon} Prikry type forcing.  In the jargon,
collapses which enforce the nonexistence of special trees are \emph{interleaved}
with the Prikry points.  The main difficulty of the paper comes from controlling
new collapses as Prikry points are added.  This is typically done by leaving
some gaps between the Prikry points and the associated collapses.  In the
current work, we have no such luxury since we wish to control the combinatorics
of every regular cardinal below the cardinal which becomes singular.

The paper is organized as follows.  In Section \ref{preliminaries}, we make some
preliminary definitions most of which are standard in the study of either
singular cardinal combinatorics or compactness properties at double successors.
In Section \ref{preparation}, we describe the preparation forcing and derive the
measures that we need for the main forcing.  In Section \ref{mainforcing}, we
describe the main forcing and prove that it gives the desired cardinal
structure.  In Section \ref{schematic}, we give a schematic view of an argument
for the failure of approachability at a double successor cardinal.  In
Section \ref{finalmodel}, we prove that the extension by the main forcing gives
the desired failure of approachability.  For the double successor cardinals we
apply the scheme from the previous section and for the successor of each
singular cardinal we apply arguments from singular cardinal combinatorics.

\section{Preliminaries} \label{preliminaries}

In this section we define the combinatorial notions and forcing posets which are
at the heart of the paper.  By a theorem of Jensen \cite{jensen}, the existence
of a special $\sigma^+$-Aronszajn tree is equivalent to the existence of a
\emph{weak square sequence} at $\sigma$.

\begin{definition}  A weak square sequence at $\sigma$ is a sequence $\langle
\mathcal{C}_\alpha \mid \alpha < \sigma^+ \rangle$ such that
\begin{enumerate}
\item For all $\alpha < \sigma^+$, $1 \leq \vert \mathcal{C}_\alpha \vert \leq
\sigma$,
\item For all $\alpha < \sigma^+$ and all $C \in \mathcal{C}_\alpha$, $C$ is a
club subset of $\alpha$ of ordertype at most $\sigma$ and
\item For all $\alpha < \sigma^+$, all $C \in \mathcal{C}_\alpha$ and all $\beta
\in \lim(C)$, $C \cap \beta \in \mathcal{C}_\beta$.
\end{enumerate}
\end{definition}

If there is such a sequence, then we say that \emph{weak square holds} at
$\sigma$.

In this paper we are interested in the weaker \emph{approachability property}
isolated by Shelah \cite{shelahcard, shelahsing}.  For a cardinal $\tau$ and a
sequence $\langle a_\alpha \mid \alpha <\tau \rangle$ of bounded subsets of
$\tau$, we say that an ordinal $\gamma < \tau$ is \emph{approachable} with
respect to $\vec{a}$ if there is an unbounded $A \subseteq \gamma$ such that
$\otp(A) = \cf(\gamma)$ and for all $\beta < \gamma$ there is $\alpha < \gamma$
such that $A \cap \beta = a_\alpha$.  Using this we can define \emph{the
approachability ideal} $I[\tau]$.

\begin{definition} $S \in I[\tau]$ if and only if there are a sequence $\langle
a_\alpha \mid \alpha < \tau \rangle$ and a club $C \subseteq \tau$ such that
every $\gamma \in S \cap C$ is approachable with respect to $\vec{a}$.
\end{definition}

It is easy to see that $I[\tau]$ contains the nonstationary ideal.  We say that
\emph{the approachability property} holds at $\tau$ if $\tau \in I[\tau]$.  We
note that weak square at $\sigma$ implies the approachability property at
$\sigma^+$ and refer the interested reader to \cite{csurvey} for a proof.

In the case where $\sigma$ is a singular cardinal, the approachability property
at $\sigma^+$ is connected with the notion of good points in Shelah's PCF
theory.  We give a brief description of this in the case when $\cf(\sigma) =
\omega$, since we will make use of it in the final argument below.

Suppose that $\langle \sigma_n \mid n < \omega \rangle$ is an increasing
sequence of regular cardinals cofinal in $\sigma$.  A sequence of
functions $\langle f_\alpha \mid \alpha < \sigma^+ \rangle$ is a \emph{scale} of
length $\sigma^+$ in $\prod_{n<\omega} \sigma_n$ if it is increasing and cofinal
in $\prod_{n<\omega}\sigma_n$ under the eventual domination ordering.  A
remarkable theorem of Shelah is that there are sequences $\langle \sigma_n \mid
n <\omega \rangle$ for which scales of length $\sigma^+$ exist.

If $\vec{f}$ is a scale of length $\sigma^+$ in $\prod_{n<\omega} \sigma_n$,
then we say that $\gamma$ is \emph{good} for $\vec{f}$ if there are an unbounded
$A \subseteq \gamma$ and an $N<\omega$ such that for all $n \geq N$, $\langle
f_\alpha(n) \mid \alpha \in A \rangle$ is strictly increasing.  A scale
$\vec{f}$ is \emph{good} if there is a club $C \subseteq \sigma^+$ such that
every $\gamma \in C$ is good for $\vec{f}$.  A scale is \emph{bad} if it is not
good.  We note that approachability at $\sigma^+$ implies that all scales of
length $\sigma^+$ are good and refer the reader to \cite{csurvey} for a proof.

Each of the principles described above can be thought of as an instance of
incompactness.  In this paper we will be interested in the negation of these
properties.  We summarize the implications discussed above, but in terms of the
negations.

\begin{enumerate}
\item For all cardinals $\sigma$, the failure of approachability at $\sigma^+$
implies the failure of weak square at $\sigma$.
\item For all singular cardinals $\sigma$, there is a bad scale of length
$\sigma^+$ implies the failure of approachability at $\sigma^+$.
\end{enumerate}

In the arguments below we will either argue for the existence of a bad scale or
directly for the failure of approachability.  In particular we will have the
failure of weak square on a long initial segment of the cardinals or
equivalently the nonexistence of special Aronszajn trees on that interval.

We now define some the forcing posets which will form the collapses between the
Prikry points in our main forcing.  The poset is essentially due to Mitchell
\cite{mitchell} from his original argument for nonexistence of special
$\aleph_2$-Aronszajn trees.  We use a more flexible version of this poset as
described by Neeman \cite{neeman}.

\begin{definition} Let $\rho < \sigma < \tau \leq \eta$ be regular cardinals and
let $\mathbb{P}=\Add(\rho, \eta)$.  Let $\mathbb{C}(\mathbb{P},\sigma,\tau)$ be
the collection of partial functions $f$ of size less than $\sigma$ whose domain
is a set of successor ordinals contained in the interval $(\sigma,\tau)$ such
that for all $\alpha \in \dom(f)$, $f(\alpha)$ is a $\mathbb{P}\restrict
\alpha$-name for an element of $\Add(\sigma,1)_{V[\P\restrict \alpha]}$.  We
order the poset by $f_1 \leq f_2$ if and only if $\dom(f_1) \supseteq \dom(f_2)$
and for all $\alpha \in \dom(f_2)$, $\Vdash_{\P \restrict \alpha} f_1(\alpha)
\leq f_2(\alpha)$.  \end{definition}

Note that such posets are easily `enriched' in the sense of Neeman to give
Mitchell like collapses.  The \emph{enrichment} of $\mathbb{C}(\P,\sigma,\tau)$
by $\mathbb{P}$ is the poset defined in the generic extension by $\P$ with the
same underlying set as $\mathbb{C}(\mathbb{P},\sigma,\tau)$, but with the order
$f_1 \leq f_2$ if and only if $\dom(f_1) \supseteq \dom(f_2)$ and there is $p$
in the generic for $\P$ such that for all $\alpha \in \dom(f_2)$, $p\restrict
\alpha \Vdash f_1(\alpha) \leq f_2(\alpha)$.  If $P$ is $\P$-generic we will
write $\mathbb{C}^{+P}$ for the enrichment of $\mathbb{C}$ by $P$.  We will drop
the $P$ and just write $\C^+$ when it is clear from context which $P$ is
required.  Note that the poset $\mathbb{C}(\mathbb{P},\sigma, \tau)$ is
$\sigma$-closed, $\tau$-cc if $\tau$ is inaccessible and collapses every regular
cardinal in the interval $(\sigma,\tau)$ to have size $\sigma$.  Hence it makes
$\tau$ into $\sigma^+$.

For notational convenience we make the following definition.

\begin{definition} Let $\rho < \sigma < \tau$ be cardinals.  If $\mathbb{P} =
\Add(\rho,\tau)$ and $\mathbb{C} = \mathbb{C}(\mathbb{P}, \sigma, \tau)$, then
we write $\M(\rho,\sigma, \tau)$ for $\P * \C^+$.
\end{definition}

We note that $\M(\rho,\sigma,\tau)$ essentially gives the family of
Mitchell-like posets defined in Abraham \cite{abraham}.

We also need two facts about term forcing.  For more complete presentation of
term forcing we refer the interested reader to \cite{cummingshandbook}.  For
completeness we sketch the relevant definitions.  Suppose that $\P$ is a poset
and $\dot{\Q}$ is a $\P$-name for a poset.  Let $\mathcal{A}(\P,\dot{\Q})$ be
the set of $\P$-names $\dot{q}$ which are forced to be elements of $\dot{\Q}$
with the order $\dot{q}_1 \leq \dot{q}_2$ if $\Vdash_{\P}$ ``$\dot{q}_1 \leq
\dot{q}_2$ in $\dot{\Q}$".

We note that $\P \times \mathcal{A}(\P,\dot{\Q})$ induces a generic for
$\P*\dot{\Q}$ by taking the upwards closure in the order on $\P*\dot{\Q}$.  This
means that the extension by any two step iteration has an outer model which is
the extension by a product.  The $\C$ posets above can be seen as a kind of term
forcing and $\P*\C^+$ can be seen as a kind of two step iteration.  We will use
this idea extensively in the proof.

In certain cases, the term poset can be realized as a nice poset in $V$.

\begin{fact} \label{cohenterms} If $\kappa^{<\kappa}=\kappa$, $\lambda \geq
\kappa$ and $\P$ is a $\kappa$-cc poset, then
$\mathcal{A}(\P,\Add^{V^\P}(\kappa,\lambda))$ is isomorphic to
$\Add(\kappa,\lambda)$. \end{fact}

If $\A$ is an iteration, then we can define a term ordering on $\A$ by $a \leq
a'$ if the support of $a$ contains the support of $a'$ and for every $\alpha$ in
the support of $a'$, it is forced by $\A \restrict \alpha$ that $a(\alpha) \leq
a'(\alpha)$.  We call this poset $\mathcal{A}(\A)$.  It is straightforward to
see that $\mathcal{A}(\A)$ induces a generic for $\A$.

It is straightforward to see that the poset $\mathcal{A}(\A)$ has many of the
same properties at the iteration.  We will need the following fact.

\begin{fact} \label{term-iteration} Let $\kappa$ be a $2$-Mahlo cardinal.
Suppose that $\A$ is an Easton support iteration of length $\kappa$ where the
set of nontrivial coordinates in the iteration is a stationary set $S$.  If
for every $\alpha \in S$, $\alpha$ is Mahlo and it is forced by $\A \restrict
\alpha$ that $\A(\alpha)$ is $\alpha$-closed, then $\mathcal{A}(\A)$ is
$\kappa$-cc and preserves the Mahloness of $\kappa$. \end{fact}

\section{The preparation forcing} \label{preparation}

We work in a model $V$ and let $\kappa$ be a supercompact cardinal. We assume
that there is an increasing sequence $\langle \kappa_i \mid i \leq \omega+3
\rangle$ such that $\kappa_0=\kappa$, $\kappa_\omega = \sup \kappa_i$,
$\kappa_{\omega+1}=\kappa_\omega^+$ and for all $i$ (except $\omega$)
$\kappa_{i+1}$ is the least Mahlo cardinal above $\kappa_i$.  For simplicity we
set $\theta=\kappa_{\omega+3}$.  For inaccessible $\alpha$, we define $\langle
\alpha_i \mid i \leq \omega+3 \rangle$ and $\alpha_{\omega+3}=\theta_\alpha$ as
we did for $\kappa$.

 We fix a supercompactness measure $U$ on $\mathcal{P}_\kappa(\theta)$ and let
$j:V \to M$ be the ultrapower map.   It is straightforward to show that there is
a set $Z \subseteq \kappa$ such that $\kappa \in j(Z)$ and for every $\gamma \in
Z$, $\gamma$ is $\gamma_{\omega+1}$-supercompact and closed under
the function $\alpha \mapsto \theta_\alpha$.  In particular we have $j(\alpha
\mapsto \alpha_i)(\kappa) = \kappa_i$ for all $i \leq \omega+3$.

We define an iteration with Easton support where we do nontrivial forcing at
each $\alpha \in Z$.  Suppose that we have defined $\mathbb{A} \upharpoonright
\alpha$ for some $\alpha \in Z$.  At stage $\alpha$ we force with $(\P_0(\alpha)
* \C_0^+(\alpha)) \times (\P_1(\alpha) * \prod_{0<i\leq\omega+1} \C_i^+(\alpha))
 \times \Add^{V[\A\upharpoonright \alpha]}(\alpha_1,\theta_\alpha^+ \setminus
\theta_\alpha)$ where

\begin{enumerate}
\item $\P_0(\alpha) = \Add(\alpha_0,\alpha_2)$ as computed in
$V[\mathbb{A}\upharpoonright \alpha]$,
\item $\C_0(\alpha) = \C(\P_0(\alpha),\alpha_1,\alpha_2)$,
\item $\P_1(\alpha) = \Add(\alpha_1,\theta_\alpha)$ as computed in $V[\A
\restrict \alpha]$ and
\item for all $i$ with $0 < i \leq \omega+1$, $\mathbb{C}_i(\alpha) =
\mathbb{C}(\P_1(\alpha),\alpha_{i+1},\alpha_{i+2})$.

\end{enumerate}

We take the product $\prod_{0<i\leq \omega+1} \C_i^+(\alpha)$ with full support.

For ease of notation we let $\A = \A \upharpoonright \kappa$, $\P =
\P_0(\kappa)\times\P_1(\kappa)$ and in the extension by $\P$ we let $\C^+ =
\prod_{i \leq \omega+1} \C_i^+(\kappa)$.  We will now lift $j$ preserving its
large cardinal properties precisely with one further addition.  Let $G$ be
$\A$-generic and let $H = (H_0 *H_1) \times H_2$ be $(\P * \C^+) \times
\Add(\kappa_1,\theta^+ \setminus \theta)$-generic over $V[G]$.

\begin{lemma}\label{lift} In $V[G*H]$, there is a generic $G^**H^*$ for
$j(\mathbb{A} *(\P*\C^+) \times \Add^{V[\A]}(\kappa_{1},\theta^+ \setminus
\theta))$ such that $j$ extends to $j:V[G*H] \to M[G^**H^*]$ witnessing that
$\kappa$ is $\theta$-supercompact and for all $\gamma < j(\kappa_{1})$, there is
a function $f:\kappa_{1}\to \kappa_{1}$ such that $j(f)(\sup
j``\kappa_{1})=\gamma$. \end{lemma}

\begin{proof} Much of the proof is standard, so we sketch some parts and give
details where important.  To construct $G^*$ we note
\begin{enumerate}
\item $j(\A \upharpoonright \kappa) \upharpoonright \kappa + 1 = \A
\upharpoonright \kappa * ( (\P*\C^+) \times \Add(\kappa_{1},\theta^+
\setminus \theta))$ and
\item the poset $j(\A)/G*H$ is $\theta^+$-closed in $V[G*H]$ and
has just $\theta^+$ antichains in $M[G*H]$.
\end{enumerate}
Standard arguments allow us to build a generic for the tail forcing and thus
form $G^*$ which is generic for $j(\A)$ over $M$ and such that $j$ lifts to
$j:V[G] \to M[G^*]$.

By closure properties of $M$, $j``(H_0*H_1) \in M[G^*]$ and is a directed set
there of cardinality $\theta$.  Moreover, $j(\P*\C^+)$ is $j(\kappa)$-directed
closed and hence we can find a master condition for $j``H_0*H_1$.  Another
routine counting of antichains allows us to build a generic $H_0^**H_1^*$ for
$j(\P*\C^+)$ over $M[G^*]$ containing our master condition.  At this point we
can lift $j$ to $j:V[G*(H_0*H_1)] \to M[G^**(H_0^**H_1^*)]$.

The final difficulty is to lift to the extension by $\Add(\kappa_{1},\theta^+
\setminus \theta)$.  It is here that we will control the values of $j(f)(\sup
j``\kappa_{1})$ for generic functions $f$ from $\kappa_{1}$ to
$\kappa_{1}$.  In preparation for this step we let $\langle \eta_\gamma \mid
\gamma \in \theta^+ \setminus \theta \rangle$ be an enumeration of $j(\kappa_{1})$.

By the usual argument counting antichains we can construct a generic $I$ for
$j(\Add(\kappa_{1},\theta^+ \setminus \theta))$ over $M[G^**H_0^**H_1^*]$.  For
each $\gamma$, let $I^\gamma$ be the restriction of $I$ to coordinates below
$j(\gamma)$.  Let $H^\gamma$ be the natural modification of $I^\gamma$ which
contains the condition $p_\gamma = \bigcup_{p \in j``H_2} p \upharpoonright
j(\gamma) \cup \{ ((j(\beta),\sup j``\kappa_{1}),\eta_\beta) \mid \beta < \gamma
\}$.  Since we have only changed a small number of coordinates, $H^\gamma$
remains generic.

Note that $j$ is continuous at $\theta^+$, so that
$j(\Add(\kappa_{1},\theta^+ \setminus \theta)) = \bigcup_{\gamma < \theta^+}
\Add(\kappa_{1}, j(\gamma) \setminus j(\theta))$.  Moreover the sequence of
$H^\gamma$ is increasing and hence $\bigcup_\gamma H^\gamma$ is generic for
$j(\Add(\kappa_{1},\theta^+ \setminus \theta))$ over $M[G^**H_0^**H_1^*]$ and
compatible with $j$.  This is enough to finish the proof. \end{proof}

For each $n<\omega$ we can derive a supercompactness measure $U_n$ on
$\mathcal{P}_\kappa(\kappa_n)$ from $j$ in $V[G*H]$.  We let $j_n:V[G*H] \to
M_n$ be the ultrapower map and let $k_n: M_n \to M[G^**H^*]$ be the factor map.
By the way that we lifted $j$, for $n \geq 1$ we have $j(\kappa_{1})+1
\subseteq \ran(k_n)$ and hence $\crit(k_n) > j(\kappa_{1})$.  This property
is essential in the arguments below.

We end the section with an proposition which will help us prove that the
relevant cardinals are preserved by the final forcing.

\begin{proposition} \label{outermodel}  For every $\alpha \in Z \cup
\{\kappa\}$ and every $n \leq \omega+3$ where $n$ is not $\omega$ or $\omega+1$,
there are an outer model $W$ of $V[G*H]$ and posets $\bar{\Y}$ and $\hat{\Y}$ such that
\begin{enumerate}
\item $W$ is an extension of $V$ by $\bar{\Y} \times \hat{\Y}$,
\item $\bar{\Y}$ is $\alpha_n$-Knaster,
\item $\hat{\Y}$ is $\alpha_n$-closed and
\item there is a generic for $\bar{\Y}$ in $V[G*H]$.
\end{enumerate}
\end{proposition}

\begin{proof}Let $\alpha$ and $n$ be as in the proposition.  We work through a
few cases.  First suppose that $n=0$.  It is
straightforward to see that the iteration can be broken up as $\A
\upharpoonright \alpha$ which is $\alpha$-cc followed by the rest of the
iteration which is forced to be $\alpha$-closed.  So we can set $\bar{\Y}$ to be
$\A \upharpoonright \alpha$ and $\hat{\Y}$ to be the poset of $\A
\upharpoonright \alpha$-names for elements of the rest of the iteration.

Suppose that $n=1$.  Then there is an outer model of $V[G*H]$ where we have a
generic for $\mathcal{A}(\mathbb{A}\upharpoonright \alpha, \P_1(\alpha) \times
\prod_{i \leq \omega+1} \C_i) \times \mathcal{A}(\mathbb{A}\upharpoonright
\alpha +1, \dot{\A} \upharpoonright [\alpha+1, \kappa+1))$.  Note that this
poset is $\alpha_1$-closed in $V$. So we take it to be $\hat{\Y}$ and set
$\bar{\Y}$ to be $\A \upharpoonright \alpha * \P_0(\alpha)$.

Suppose that $n \geq 2$.  Then there is an outer model of $V[G*H]$ where we have
a generic for $\mathcal{A}(\mathbb{A}\upharpoonright \alpha, \prod_{n-1 \leq i
\leq \omega+1} \C_i) \times \mathcal{A}(\mathbb{A}\upharpoonright \alpha +1,
\dot{\A} \upharpoonright [\alpha+1, \kappa+1))$.  Note that this poset is
$\alpha_n$-closed in $V$.  So we can take it to be $\hat{\Y}$ and set $\bar{\Y}$
to be $\A \upharpoonright \alpha * \P_0(\alpha)*\C_0^+(\alpha) \times
\P_1(\alpha) * \prod_{1 \leq i <n-1}\C_i^+$. \end{proof}

It is immediate that each such $\alpha_n$ is preserved and for all $n \geq 1$,
$\alpha_{n+1} = \alpha_n^+$ in $V[G*H]$.  Further standard arguments using the
above proposition show that $\alpha_{\omega}$ and $\alpha_{\omega+1}$ are
preserved for all $\alpha \in Z \cup \{\kappa\}$.

\section{The main forcing} \label{mainforcing}

In order to define a diagonal Prikry forcing we define a collection of
Mitchell-like collapses which will go between the Prikry points.  Let
$\Q(\alpha,\beta) = \Q^0(\alpha,\beta) \times \Q^1(\alpha,\beta)$ where
\begin{align*}
\Q^0(\alpha,\beta) = & \Add(\alpha_{\omega+2},\beta) *
\C^{+}(\Add(\alpha_{\omega+2},\beta),\alpha_{\omega+3},\beta) \\
\Q^1(\alpha,\beta) = & \Add(\alpha_{\omega+3},\beta_1) *
\C^+(\Add(\alpha_{\omega+3},\beta_1), \beta, \beta_1)
\end{align*}
Extremely important to the construction is that we take all these posets as
defined in $V$.  The intention is to force $\beta$ to become
$\alpha_{\omega+3}^+$ and $\beta_1$ to be $\beta^+$.  We will be sloppy and
write $\Q(x,y)$ for $\Q(\kappa_x,\kappa_y)$.

We are now ready to define the diagonal Prikry poset.  Let $Z_n$ be the set of
$x$ in $\mathcal{P}_\kappa(\kappa_n)$ such that $\kappa_x = x \cap \kappa \in
Z$ and $\otp(x)$ is $\alpha_n$ where $\alpha = \kappa_x$.  Clearly $Z_n \in
U_n$.  For $n<m$, $x \in Z_n$ and $y \in Z_m$, we write $x \prec y$ for $x
\subseteq y$ and $\vert x \vert < \kappa_y$.

\begin{definition} Let $\R$ be a poset where conditions are of the form
\[\langle
q_0,x_0,q_1,x_1, \dots q_{n-1},x_{n-1},f_{n},F_{n+1},F_{n+2}, \dots \rangle\]
such that

\begin{enumerate}
\item $n\neq 0$ implies $q_0 \in
\Coll(\omega,\alpha_\omega)$ where $\alpha = \kappa_{x_0}$.
\item for all $i<n$, $x_i \in Z_i$ and $\vec{x}$ is $\prec$-increasing,
\item for all $i \in [1,n)$, $q_i \in \Q(x_{i-1},x_i)$,
\item There is a sequence of measure one sets $\langle A_i \mid i \geq n
\rangle$ such that $\dom(f_{n}) = A_{n}$ and for all $i \geq n+1$, $\dom(F_i) =
A_{i-1} \times A_{i}$,
\item $n=0$ implies for all $x \in A_n$, $f_n(x) \in \Coll(\omega,\alpha_\omega)$ where
$\alpha=\kappa_x$ and otherwise for
all $x \in A_n$, $f_{n}(x) \in \Q(x_{n-1},x)$
\item for all $i \geq n+1$ and $(x,y) \in \dom(F_i)$, $F_i(x,y) \in \Q(x,y)$.
\end{enumerate}

For a condition $p \in \P$ we adorn each part of $p$ with a superscript to
indicate its connection with $p$.  For example $q_0^p, x_0^p$ etc.  We also let
$\ell(p) = n$ denote the length of the condition $p$.  For $p,r \in \P$ we say
that $p \leq r$ if $\ell(p) \geq \ell(r)$ and
\begin{enumerate}
\item $\vec{x}^p$ end extends $\vec{x}^r$ and new points come from measure one
sets of $r$,
\item $\vec{q}^p \upharpoonright \ell(r) \leq \vec{q}^r$
\item $q_{\ell(r)}^p \leq f_{\ell(r)}^r(x_n^p)$
\item for all $i \in [\ell(r)+1, \ell(p))$, $q_i^p \leq F_i^r(x_{i-1},x_i)$
\item for all $i \geq \ell(p)$, $A_i^p \subseteq A_i^r$,
\item for all $x \in A_{\ell(p)}$, $f_{\ell(p)}^p(x) \leq
F_{\ell(p)}^r(x_{\ell(p)}^p,x)$ and 
\item for all $i \geq \ell(p)+1$ and all $(x,y) \in A^p_{i}\times A^p_{i+1}$, $F_i^p(x,y) \leq F_i^r(x,y)$.
\end{enumerate}
\end{definition}

We fix some terminology:

\begin{enumerate}
\item We call the lower part or stem of a condition any sequence of the form
$\langle q_0,x_0,q_1, \dots q_{n-1},x_{n-1} \rangle$.  Note that we have not included
$f_{n}$.
\item If $r$ is a condition, then we write $\stem(r)$ to denote the stem of $r$.
\item If $s =\langle q_0,x_0,q_1, \dots q_{n-1},x_{n-1} \rangle$ is a stem, then
we let $\top(s) = x_{n-1}$.
\item We call the one variable functions $f_n^p$ the $f$-part of $p$.
\item We call the sequence of two variable functions $\vec{F}^p$ the upper part
or constraint of $p$.

\end{enumerate}

\begin{remark}  For all $n<\omega$, there are $\kappa_{n-1}$ stems of length
$n$. \end{remark}

For each $i > 1$, the class of a function $F_i$ as above modulo
$U_{i-1}\times U_{i}$ is a member of $\Q_i = \Q(\kappa,j_{i}(\kappa))$ as
computed in the (external) ultrapower of $V$ by $U_{i-1}\times U_{i}$.  This
forcing is a product of Mitchell-like collapses below $j_{i-1}(\kappa_1)$ which
is the least (formerly) Mahlo cardinal above $j_{i-1}(\kappa)$.

We fix some notation.  For $i > 1$, $M_{i-1}^{i} = \Ult(V[G*H],U_{i-1} \times
U_{i}))$ and $j_{i-1}(M[G^**H^*)) = \Ult(M_{i-1},j_{i-1}(U))$.  By standard
arguments there are elementary embeddings, $k: M_{i-1}^{i} \to
j_{i-1}(M[G^**H^*])$ and $\hat{k}: j_{i-1}(M[G^**H^*]) \to \Ult(V[G*H],U \times
U)$.

\begin{claim} $\crit(\hat{k} \circ k) > j(\kappa_1)$ \end{claim}

\begin{proof} Note that $k$ is given by $j_{i-1}(k_{i})$ and hence has critical
point above $j_{i-1}(j(\kappa_1)) \geq j(\kappa_1)$ by the way that we extended
the embedding $j$.  Further note that $\hat{k}$ is given by the restriction of
$k_i$ to the domain of $\hat{k}$ and hence has critical point above
$j(\kappa_1)$ by the way that we extended the embedding $j$.  It follows that
the composition has critical point above $j_1(\kappa)$.  \end{proof}

\begin{corollary} For $i > 1$, $\Q_i$ is $\Q(\kappa,j(\kappa))$ as
computed in $M$. \end{corollary}

\begin{proof} Recall that $\Q_i$ is $\Q(\kappa,j_{i-1}(\kappa))$ as computed in
the external ultrapower of $V$ by $U_{i-1}\times U_i$.  Of course using the high
critical point of $k_{i-1}$, $j_{i-1}(\kappa)=j(\kappa)$.  Moreover, by the
previous lemma, $(\hat{k} \circ k)(\Q_i) = (\hat{k} \circ k)``\Q_i$ and
$(\hat{k} \circ k) \upharpoonright \Q_i$ is the identity map.  It follows that
$\Q_i$ is $\Q(\kappa,j(\kappa))$ as computed in $\Ult(V,U \times U)$.  However
this ultrapower is highly closed inside $M = \Ult(V,U)$ and the conclusion
follows. \end{proof}

So we have shown that each $\Q_i$ for $i \geq 1$ is in fact the same forcing and
we drop the dependence on $i$ and just write $\Q$.  In the course of the proof
of the Prikry lemma below the forcing $\Q$ will be represented in different ways
in different ultrapowers.  For clarity record the following remark.

\begin{remark} \label{qinups} For each $i>1$, $\Q$ is isomorphic to the set $\{
f \in M_{i-1} \mid \dom(f) \in j_{i-1}(U_i)$ and for all $x \in \dom(f)$, $f(x)
\in \Q(\kappa,\kappa_x)$ as computed the external ultrapower of $V$ by
$U_{i-1}\}$ with the natural ordering.  These functions $f$ are the ones that
represent elements of $\Q_i$ in the ultrapower by $j_{i-1}(U_i)$.
\end{remark}

Next we work towards showing that forcings $\Q$ and $\Q(\alpha,\beta)$ are
well-behaved.

\begin{claim} \label{dist0} The full support power $\Q^\omega$ is
$<\kappa_{\omega+2}$-distributive in $V[G*H]$. \end{claim}

\begin{proof} $\Q$ is defined in $M$ (hence $V$) and is
$\kappa_{\omega+2}$-closed in $V$ by the closure of $M$. Hence $\Q^\omega$ is
$\kappa_{\omega+2}$ closed in $V$.  Let $W$, $\bar{\Y}$ and $\hat{\Y}$ be the
outer model and posets from Proposition \ref{outermodel} applied to
$\kappa_{\omega+2}$.  By Easton's lemma, every $<\kappa_{\omega+2}$-sequence
from $W^{\Q^\omega}$ is in $V^{\bar{\Y}}$, but $V[G*H]$ contains a generic for
$\bar{\Y}$.  So we are done.\end{proof}

A similar argument establishes the following claim.

\begin{claim}\label{dist} For all $\alpha<\beta \leq \kappa$, $\Q(\alpha,\beta)$
is $< \alpha_{\omega+2}$-distributive in $V[G*H]^{\Q^\omega}$. \end{claim}

\begin{remark} \label{smalldist}  It is clear from the proof above that the
conclusion of the previous claim holds in any forcing extension by a poset from
$V$ of size less than $\alpha_{\omega+1}$.  The small poset can be included in
$\bar{\Y}$.  \end{remark}

We pause here to prove that the posets $\Q(\alpha,\beta)$ have the desired
effect on the cardinals between $\alpha$ and $\beta$.

\begin{proposition} \label{productpreservation}  In any extension of $V[G*H]$ by
a poset of size less than $\alpha_{\omega+1}$ from $V$, $\Q(\alpha,\beta)$ is
$\beta_1$-cc, preserves the cardinals $\alpha_{\omega+3}$ and $\beta$ and forces
$ \beta = \alpha_{\omega+3}^+$ and $\beta_1 = \beta^+$. \end{proposition}

\begin{proof}  Let $\Y$ be a poset of size less than $\alpha_{\omega+1}$ in
$V$.  First we consider the outer model $W$ and posets $\bar{\Y}$ and $\hat{\Y}$
from Proposition \ref{outermodel} applied to $\beta_1$.  In $V$, the product $\Y
\times \bar{\Y} \times \Q(\alpha,\beta)$ is $\beta_1$-Knaster.  By Easton's
lemma, it follows that $\Q(\alpha,\beta)$ is $\beta_1$-cc in $W^{\Y}$ and hence
in $V[G*H]^{\Y}$.

Second we consider the outer model $W$ and posets $\bar{\Y}$ and $\hat{\Y}$ from
Proposition \ref{outermodel} applied to $\beta$.  Using the definition of the
$\C$ posets, there is an outer model of $V^{\Q(\alpha,\beta)}$ which is an
extension by the partial order $\Q^0(\alpha,\beta) \times
\Add(\alpha_{\omega+3}, \beta_1) \times \C(\Add(\alpha_{\omega+3},
\beta_1),\beta,\beta_1)$.  Hence $W^{\Q(\alpha,\beta)}$ is contained in a
generic extension by the product of $\Y \times \bar{\Y} \times
\Q^0(\alpha,\beta) \times \Add(\alpha_{\omega+3}, \beta_1)$ and $\hat{\Y} \times
\C(\Add(\alpha_{\omega+3}, \beta_1),\beta,\beta_1)$.  The first piece is
$\beta$-cc and the second is $\beta$-closed.  Hence $\beta$ is preserved in
$V[G*H]^{\Y \times \Q(\alpha,\beta)}$.

The argument that $\alpha_{\omega+3}$ is preserved is similar to the argument
that $\beta$ is preserved, but we split up $\Q^0(\alpha,\beta)$ instead of
$\Q^1(\alpha,\beta)$ and incorporate $\Q^1(\alpha,\beta)$ into the closed part.

The argument that cardinals in the intervals $(\alpha_{\omega+3},\beta)$ and
$(\beta, \beta_1)$ are collapsed is standard for Mitchell type posets.
\end{proof}

\begin{claim} \label{qomegapres}  Forcing with $\mathbb{Q}^\omega$ over $V[G*H]$
preserves cardinals up to $\kappa_{\omega+3}$.  \end{claim}

\begin{proof} By Claim \ref{dist0}, it is enough to show that
$\kappa_{\omega+3}$ is preserved.  Recall that $\mathbb{Q}$ is computed in an
ultrapower of $V$ which is closed under $\theta= \kappa_{\omega+3}$-sequences.
In particular, it can be written as the projection of a product
$\Add(\kappa_{\omega+2},j(\kappa))$ which is $\kappa_{\omega+3}$-cc in $V$ and
$\mathbb{C}(\Add(\kappa_{\omega+2},j(\kappa)),\kappa_{\omega+3},j(\kappa))
\times \Q^1(\kappa,j(\kappa))$ which is $\kappa_{\omega+3}$-closed in $V$.
Since the iteration to add $G*H$ is $\kappa_{\omega+3}$-cc, we have that
$\kappa_{\omega+3}$ is preserved when we force with $\Q^\omega$. \end{proof}

We are now ready to prove the Prikry lemma.  The main elements come from a
combination of \cite{jamesthesis} and \cite{su1}.  Suppose that we force with
$\mathbb{Q}^\omega$ to obtain $ \vec{K} = \langle K_n \mid n > 1 \rangle$.
We let $\bar{\R}$ be the set of conditions $r$ such that $\langle [F_n^r] \mid n
\geq \max(\ell(p),1) \rangle \in \prod_{n \geq \max(\ell(p),2)}K_n$ ordered as a suborder of $\R$.  The following claims are
straightforward.

\begin{claim} In $V[G*H]^{\mathbb{Q}^\omega}$, $\bar{\R}$ is
$\kappa_{\omega}$-centered below any condition of length at least one.
\end{claim}

Note that in $\bar{\R}$, conditions of length at least one with
the same stem and (equivalence class of) $f$-part are compatible and there are
at most $\kappa_\omega$ such pairs.

\begin{claim} $\Q^\omega *\bar{\R}$ projects to $\R$. \end{claim}

This claim allows us to prove the Prikry lemma for $\bar{\mathbb{R}}$ in place
of $\mathbb{R}$.

\begin{lemma}  Work in $V[G*H][\vec{K}]$. Let $r \in \bar{\mathbb{R}}$ be a
condition of length at least $1$ and $\varphi$ be a statement in the forcing
language.  There is an $r^* \leq^* r$ which decides $\varphi$. \end{lemma}

It follows immediately that $\mathbb{R}$ itself has the Prikry property, since
the projection from the previous claim fixes the length of a condition.  We
break the proof into many rounds.

\begin{claim} There is an $r_0 \leq^* r$ such that for all $p \leq r_0$ if $p$
is at least a one point extension of $r_0$ and it decides $\varphi$, then
there is an upper part $\vec{F}$ such that $\stem(p) \frown
F_{\ell(p)}^{r_0}(x_{\ell(p)-1}^p) \frown \vec{F}$ decides $\varphi$.
\end{claim}

\begin{proof} We go by induction on the length of extensions of $r$.  Let
$r^{\ell(r)} =r$.  Assume that we have constructed $r^n$ for some $n<\omega$.
Work in the ultrapower $M_n$ and consider conditions of length $n+1$ of the form

\[ s \frown (q,j_n``\kappa_n) \frown (f^+, \vec{F}^+) \leq \stem(r^n) \frown
(j_n(f^{r^n}),j_n(\vec{F}^{r_n})) \]

Note that the collection of possible stems $s$ here is exactly the collection of
$j_n$ pointwise images of stems of length $n$ from $\mathbb{R}$, since $\top(s)
\prec j_n ``\kappa_n$ implies that $\top(s) = j_n``x$ for some $x$ in
$\mathcal{P}_{\kappa}(\kappa_{n-1})$.  Further, $q \in
\mathbb{Q}(\kappa_x,\kappa)$ where $x$ is as above, which has size $\kappa_1$.
It follows that there are at most $\kappa_{n-1}$ many such pairs $s,q$.

Note that $f^+ \leq j_n(F_{n+1}^{r^n})(j_n``\kappa_n)$ and in particular by
Remark \ref{qinups} there is a function $F^*$ such that $f^+ =
j(F^*)(j_n``\kappa_n)$.  For each $s$ and $q$, the set $D_{s,q}$ defined as
\begin{align*}  \{ f^+ \mid \text{if } \exists \vec{F}^+, \exists f^{++} \leq
f^+ \text{ such that } s \frown (q,j_n``\kappa_n) \frown (f^{++},\vec{F}^+)
\text{ decides } \varphi \\
 \text{ then } \exists \vec{F}^{++} s \frown (q,j_n``\kappa_n) \frown
(f^+,\vec{F}^{++}) \text{ decides } \varphi \} \end{align*}
is dense open in $\mathbb{Q}_n$, which is just $\mathbb{Q}$.  Moreover it can be
defined in the model $V[G*H][\langle K_i \mid i > n \rangle]$
where $\mathbb{Q}$ is $<\kappa_{\omega+2}$-distributive.  So the set $D =
\bigcap_{s,q} D_{s,q}$ is dense open in $\mathbb{Q}$.  Since $D$ can be defined
in $V[G*H][\langle K_i \mid i > n \rangle]$ and by the product
lemma $K_n$ is generic for $\mathbb{Q}$ over this model, we can take a function
$F_{n}$ so that $[F_{n}]_{U_{n-1}\times U_{n}} \in K_{n}$ and
$j_n(F_{n})(j_n``\kappa_n) \in D$.  We can assume that $F_{n} \leq
F_{n}^{r^n}$ on a $U_{n} \times U_{n+1}$ large set.

By Los' theorem the set $A_n$ given by
\begin{align*} \{ x \mid \forall s \text{ if } \top(s) = x, \exists \vec{F}^+ \exists f^{++}
\leq F_n(x),\ s \frown (f^{++},\vec{F}^+) \text{ decides } \varphi \\
 \text{ then } \exists \vec{F}^{++},\ s \frown (F_n(x),\vec{F}^{++}) \text{
decides } \varphi \} 
\end{align*}
is $U_n$ measure one.  We define $r^{n+1}$ by refining
$F_n^{r^n}$ to $F_n\upharpoonright A$ and leaving the rest of $r^n$ unchanged.

We let $r_0$ be a lower bound for $\langle r^n \mid n \geq \ell(r) \rangle$.  It
is straightforward to check that $r_0$ satisfies the conditions of the claim.
\end{proof}

\begin{claim} There is an $r_1 \leq^* r_0$ such that if $p \leq r_1$ is at least
a one point extension and it decides $\varphi$, then $\stem(p) \frown
F_{\ell(p)-1}^{r_1}(x_{\ell(p)-1}^p) \frown \vec{F}^{r_1}$ decides $\varphi$.
\end{claim}

\begin{proof} We collect witnesses to the previous claim.  Let $s$ be a stem
which is at least a one point extension of $r_0$.  If possible we select an
upper part $\vec{G}^s$ witnessing that the condition with stem $s$ from the
previous claim decides $\varphi$.  Using the distributivity of
$\mathbb{Q}^\omega$ (in particular that the sequence of generics $\vec{K}$ is
closed), for each $n \geq \ell(r_0)+1$ we can find $\vec{G}^n$ such that for all
$k \geq n$ and all $s$ of length $n$, $[G^n_k] \leq [G_k^s]$.  It is straight
forward to find a sequence $\langle G_k \mid k \geq \ell(r_0)+1 \rangle$ such
that for all $n \geq \ell(r_0)+1$, $[G_k] \leq [G_k^n]$.

For each stem $s$ there is a sequence of measure one sets $\vec{A}^s$ on which
$\vec{G}\upharpoonright [\ell(s),\omega)$ is below $\vec{G}^s$.  We can assume
that the sequence $\vec{A}^s$ is contained (pointwise) in the sets which form
the domains of $\vec{G}^s$.  By a standard argument there is a sequence of
measure one sets $\langle A_n \mid n \geq \ell(r_0)+1 \rangle$ such that for all
$x \in A_n$ if $s$ is a stem with $s \prec x$, then $x \in A_n^s$.

We obtain $r_1$ by for all $n \geq \ell(r_0)+1$ restricting $F_n^{r_0}$ to
$G_n\upharpoonright A_n$.  It is straightforward to check that this condition
satisfies the claim. \end{proof}

\begin{claim} There is a condition $r_2 \leq^* r_1$ such that if $p \leq r_2$ is
at least a two point extension and it decides $\varphi$, then 
\[ \stem(p) \restrict \ell(p)-2 \frown
(F_{\ell(p)-2}^{r_2}(x_{\ell(p)-2}^p, x_{\ell(p)-1}^p), x_{\ell(p)-1}) \frown
( F_{\ell(p)) -1}^{r_2}(x_{\ell(p)-1}^q),\vec{F}^{r_2})\]
 decides $\varphi$.
\end{claim}

\begin{proof}  We work by induction on the lengths of possible extensions of
$r_1$ of at least two points.  Let $r_1 = r^{\ell(r)}$ and assume that we have
constructed $r^n$ for some $n$.  We work in $M_n^{n+1}$ and consider conditions
of length $n+2$ of the form:
\[ s \frown  (j_n^{n+1}``\kappa_n, q,k_n^{n+1}``(j_n(\kappa_{n+1}))) \frown
(j_n^{n+1}(f^{r^n}),j_n^{n+1}(\vec{F}^{r^n}) \]
Note that $q \in \mathbb{Q}$ and $s$ is a stem of length $n$ from
$\mathbb{R}$, since $\top(s) \prec j_n^{n+1}``\kappa_n$.  It follows that there
are at most $\kappa_{n-1}$ many such pairs $s,q$.  We denote this condition
above $r_{s,q}$.

For each $s$
of length $n$, the set $D_s$
\begin{align*} \{ q \mid r_{s,q} \text{ decides } \varphi \text{ or for no extension }
q'\text{ of } q \text{ does } r_{s,q'} \text{  decide } \varphi \} \end{align*}
is dense in $\mathbb{Q}$ and defined in $V[G*H][\langle K_m \mid m > n
\rangle]$.  Using the distributivity of $\mathbb{Q}$ in this model, the set $D =
\bigcap_{ \ell(s) = n}D_s$ is dense in $\mathbb{Q}$.  So we can find a function
$F_n$ such that $[F_n] \in D$ with $F_n \leq F_n^{r^n}$.

By Los' theorem the set
\begin{align*} \{ (x,y) \mid s \frown (x,F_n(x,y),y) \frown (f^{r^n},\vec{F}^{r^n}) \text{ decides } \varphi \\ \text{ or there is no extension } q \leq F_n(x,y) \text{ which
decides } \varphi \} \end{align*}
is measure one for $U_n\times U_{n+1}$.  We fix measure one sets
$A_n^s,A_{n+1}^s$ so that every $(x,y) \in A_n^s \times A_{n+1}^s$ with $x \prec
y$ is in the above set.  Again by a standard construction we can find $A_n
\times A_{n+1}$ such that for all pairs $(x,y) \in A_n\times A_{n+1}$ if $s
\prec x \prec y$, then $(x,y) \in A_n^s \times A_{n+1}^s$.  We refine $r^n$ to
$r^{n+1}$ by replacing $F_n^{r^n}$ with $F_n\upharpoonright A_n \times A_{n+1}$.

At the end of the construction we let $r_2$ be a lower bound for $r^n$ for $n
\geq \ell(r_1)$.  It is straightforward to check that $r_2$ has the desired
property. \end{proof}

For a stem $s$ of length $n \geq \ell(r_2)+1$, we partition $A_n^{r_2}$ in to
three sets
\begin{align*}
A_s^0 &= \{ x \mid s \frown (F_n^{r_2}(\top(s),x),x) \frown
(F_{n+1}^{r_2}(x),\vec{F}^{r_2}) \Vdash \varphi \} \\
A_s^1 &= \{ x \mid s \frown (F_n^{r_2}(\top(s),x),x) \frown
(F_{n+1}^{r_2}(x),\vec{F}^{r_2}) \Vdash \lnot\varphi \} \\
A_s^2 &= \{ x \mid s \frown (F_n^{r_2}(\top(s),x),x) \frown
(F_{n+1}^{r_2}(x),\vec{F}^{r_2}) \text{ does not decide } \varphi \}
\end{align*}

Let $A_s$ be the unique set above which is $U_n$-measure one. Let $r_3$ be
obtained by restricting the measure one sets of $r_2$ to the diagonal
intersections of the $A_s$.

Let $p^* \leq r_3$ be an extension of minimal length deciding $\varphi$.  We
assume without loss of generality that it forces $\varphi$.

\begin{claim} $\ell(p^*) \leq \ell(r_3) +1$ \end{claim}

\begin{proof} Suppose not.  Then $p^*$ is at least a two point extension of
$r_3$.  Let $n = \ell(p^*)-1$ and $s$ be $\stem(p^*) \upharpoonright n$ and let
$x = x_{n}^{p^*}$.  From our previous
claims, the condition 
\[
s \frown (F_n^{r_2}(\top(s),x),x) \frown (F_{n+1}^{r_2}(x),\vec{F}^{r_2})
\]
forces $\varphi$.  It follows that $A_s = A_s^0$ and hence $s \frown
(F_n^{r_2}(\top(s)),\vec{F}^{r_2})$ forces $\varphi$, a contradiction to the
minimality of the length of $p^*$.
\end{proof}

\begin{claim} There is a direct extension of $r_3$ which decides $\varphi$. \end{claim}

\begin{proof} For each $x \in A_{\ell(r_3)}^{r_3}$, we use the distributivity of
$\Q(x_{\ell(r)-1}^r,x)$ to record a condition $q_x
\leq f_{\ell(r_3)}^{r_3}(x)$ such that for all relevant stems $s$ the condition
\[ r(s,x) = s \frown (q_x,x) \frown (F_{\ell(r_3)}^{r_3}(x),\vec{F}^{r_3}) \]
either decides $\varphi$ or no extension of $q_x$ in the above condition decides
$\varphi$.

For each stem $s$, there is a measure one set $A_s$ of $x$ which all give the
same decision above relative to $s$.  Let $A$ be the diagonal intersection of
the $A_s$.  Let $r_4$ be obtained from $r_3$ by restricting
$f_{\ell(r_3)}^{r_3}$ to the function $x \mapsto q_x$ on $A$.  Now by the
previous lemma there is a one point extension $p$ of $r_4$ which decides
$\varphi$.  Without loss of generality we assume it forces $\varphi$.  Let $s$
be the stem of $p$.  By the above construction, $r(s,x)$ forces $\varphi$ for
all $x \in A_s$.  It follows that the condition $s \frown
(f_{\ell(r_4)}^{r_4},\vec{F}^{r_3})$ decides $\varphi$. \end{proof}

This finishes the proof of the Prikry lemma.

\begin{corollary}\label{boundedsets} In the extension by $\mathbb{R}$, $\kappa =
\aleph_{\omega^2}$ and if $\kappa_n = x_n \cap \kappa$, then $\kappa_{n,i}$ is
preserved for all $i \leq \omega+3$. \end{corollary}

\begin{proof} By Remark \ref{smalldist} and Proposition
\ref{productpreservation}, it is enough to show that if $\dot{X}$ is an
$\mathbb{R}$-name for a subset of some $\mu < \kappa$ and $r$ is a condition
with $ \mu < \lambda = \kappa \cap \top(\stem(r))$, then $r$ forces $\dot{X}$ is
in the extension by $\prod_{i<\ell(r)} \mathbb{Q}(x_{i-1}^r,x_i^r)$.

By the Prikry Lemma, $\dot{X}$ is in any extension by $(\mathbb{R}, \leq^*)$
below the condition $r$.  This forcing decomposes as the product of
$\prod_{i<\ell(r)} \mathbb{Q}(x_{i-1}^r,x_i^r)$ and the forcing of upper parts
ordered by direct extension.  By an argument similar to the one from Claim
\ref{dist}, the forcing of upper parts ordered by direct extension is $< \kappa
\cap x_{\ell(r)}^r$-distributive in the extension by $\prod_{i<\ell(r)}
\mathbb{Q}(x_{i-1}^r,x_i^r)$.  So we have the desired conclusion. \end{proof}

Recall that $\Q^\omega$ is the full support product.  Clearly both $\mathbb{R}$
and $\Q^\omega$ project to $\Q^\omega/fin$.  Let $I$ be the
$\mathbb{Q}^\omega/fin$-generic induced by $\vec{K}$.

\begin{claim} $\R/I$ has the $\kappa_{\omega+1}$-Knaster property. \end{claim}

\begin{proof} Work in $V[G*H][I]$ and let $\langle r_\alpha \mid \alpha <
\kappa_{\omega+1} \rangle$ be a sequence of elements of $\mathbb{R}/I$.  We can
assume that each $r_\alpha$ has some fixed length $l$.

Let $\langle [F_i] \mid i < \omega \rangle/fin$ be a condition in
$I$ forcing this.  By the distributivity of $\mathbb{Q}^\omega/fin$ we can
assume that the $F_i$ have the property that for each $\alpha$ there is an
$n_\alpha$ such that for all $i \geq n_\alpha$, $[F_i] \leq [F_i^{r_\alpha}]$.
By passing to an unbounded subset of the $r_\alpha$, we can assume there is
$n^*$ such that $n^*=n_\alpha$ for all $\alpha < \kappa_{\omega+1}$.  Further,
extending each $r_\alpha$ if necessary we can assume that $l= \ell(r_\alpha)
\geq n^*$.  By passing to a further unbounded subset, we can assume that for all
$i < l$, $x_i^{r_\alpha} = x_i^{r_\beta}$ for all $\alpha$ and $\beta$.

Now for each $\alpha$, $r_\alpha \upharpoonright l+1$ essentially comes from the
poset $\prod_{i<l} \mathbb{Q}(x_{i-1}^r,x_i^r) \times \Q(x_{l-1},j_l``\kappa_l)$
where the latter forcing is computed in $M_l$.  This forcing has cardinality
less than $\kappa_{\omega}$ and hence we can find an unbounded set of $\alpha$
on which any $r_\alpha$ and $r_\beta$ are compatible. \end{proof}

\begin{corollary} $\mathbb{R}$ preserves the cardinals $\kappa_{\omega+1},
\kappa_{\omega+2}$ and $\kappa_{\omega+3}$.  \end{corollary}

By Claim \ref{qomegapres} all three cardinals are preserved in $V[I]$ which is
an inner model of an extension by $\Q^\omega$ and by the previous Claim they are
preserved in the extension by $\mathbb{R}$.

\section{A schematic view of arguments for the failure of approachability at
double successors} \label{schematic}

In this section we give an abstract overview of arguments for the failure the
approachability property at double successor cardinals.   Before we begin, note
that none of the cardinals and posets in this section bear any relation to those
defined elsewhere in the paper.

We begin remarking that the approachability property is upwards
absolute to models with the same cardinals.  So to prove that it fails in some model, it is enough to show that it fails
in an outer model with the same cardinals.  We formalize this in the following
remark.

\begin{remark} \label{downwardsabsolute} Suppose that $W \subseteq W'$ are
models of set theory and $\lambda$ is a regular cardinal in $W'$.  If $\lambda
\notin I[\lambda]$ in $W'$, then $\lambda \notin I[\lambda]$ in $W$.
\end{remark}



We will also use a theorem of Gitik and Krueger \cite{GK} which allows us to
preserve the failure of approachability in some outer models.

\begin{theorem} \label{GK} Suppose that $\lambda = \mu^{++}$ and $\mathbb{P}$ is
$\mu$-centered.  If $\lambda \notin I[\lambda]$ in $V$, then it is forced by
$\mathbb{P}$ that $\lambda \notin I[\lambda]$. \end{theorem}

The bulk of this section is devoted to giving an abstract view of the failure of
approachability in the extension by the Mitchell posets as described in Section
\ref{preliminaries}.  In particular we need to show that these posets have
approximation properties.

\begin{definition} [Hamkins] Let $\P$ be a poset and
$\kappa$ be a cardinal.  We say that $\mathbb{P}$ has the
$\kappa$\emph{-approximation property} if for every ordinal $\mu$ and every name
$\dot{x}$ for a subset of $\mu$, if for all $z \in
\mathcal{P}_\kappa(\mu)$ $\Vdash_\P \dot{x} \cap z \in V$, then $\Vdash_\P \dot{x} \in
V$. \end{definition}

We consider the following general situation.  Let $\rho < \sigma < \tau$ be
regular cardinals with $\tau$ Mahlo and let $\M = \M(\rho,\sigma,\tau)$ as
defined in Section \ref{preliminaries}.  Let $\X$ be a poset such that for all
$\alpha \leq \tau$, $\X$ is $\alpha$-cc in the extension by $\M\restrict
\alpha$.  (Here $\M \restrict \tau = \M$.)  Suppose that $\langle \dot{a}_\gamma
\mid \gamma < \tau \rangle$ and $\dot{C}$ are $\M \times \X$-names for witnesses
that $\tau \in I[\tau]$.  

We assume that there are a club $D$ and subforcings $\X \restrict \alpha$ of
$\X$ for $\alpha \in D \cup \{\tau\}$ such that for all $\alpha \in D \cup
\{\tau\}$, $\langle \dot{a}_\gamma \mid \gamma < \alpha \rangle$ is in the
extension by $\M \restrict \alpha \times \X \restrict \alpha$.  By the $\tau$-cc
of $\M \times \X$ we can assume that $\dot{C}$ is in $V$, so we rename it $C$.
Since $\tau$ is Mahlo, there is an inaccessible $\alpha$ in $D \cap C$.

It follows that $\alpha = \sigma^+$ in the extension by $\M \upharpoonright
\alpha \times \X \restrict \alpha$ and $\alpha \in I[\alpha]$ as witnessed by
$\langle a_\gamma \mid \gamma < \alpha \rangle$ and $C \cap \alpha$.  Since
$\alpha$ is approachable in the extension by $\M \times \X$, there is an
$\M/\M\restrict \alpha \times \X/\X\restrict \alpha$-name
$\dot{A}$ for a subset of $\alpha$ of ordertype $\sigma$ such that for all
$\delta < \alpha$, $\dot{A} \cap \delta = a_\gamma$ for some $\gamma<\alpha$.

Since forcing with $\X$ over the extension by $\M \restrict \alpha$ preserves
$\alpha$, for the failure of approachability at $\tau$ it is enough to show that
$\M/\M\restrict \alpha$ has the $\lambda$-approximation property in the
extension by $\X \times \M \restrict \alpha$ for some $\lambda \leq \sigma$.

In the arguments for the failure of approachability below the choice of
$\mathbb{X}$ will vary, but we can prove a single lemma which captures most of
the different choices.  We say that a poset $\X$ preserves the $\lambda$-cc of a
poset $\P$ if $\P$ is $\lambda$-cc in the extension by $\X$.

\begin{lemma} \label{formerapproximation} Let $\rho < \sigma < \tau$ be
cardinals and let $\M = \M(\rho,\sigma,\tau)$.  Let $\lambda \leq \sigma$ be a
cardinal and let $\X$ be a poset such that 
\begin{enumerate}
\item for all $\alpha \leq \tau$, $\X$ is $\alpha$-cc in the extension by
$\M\restrict \alpha$ and
\item $\X \simeq \bar{\X} \times \hat{\X}$ where $\bar{\X}^2$ is $\lambda$-cc
and $\hat{\X}$ is $<\lambda$-distributive and preserves the $\lambda$-cc of $\Add(\rho,\tau)$ and $\bar{\X}$.
\end{enumerate}
Then in the extension by $\M\upharpoonright \alpha \times \X$, $\M/\M\restrict
\alpha$ has the $\lambda$-approximation property.
\end{lemma}

The lemma is immediate from the proof of Lemma 4.1 of \cite{su4}.  As
stated the lemma involves a Prikry type forcing, which we can take to be
trivial.  We apply that lemma in the model $V[\hat{\X}]$ to the forcing
$\mathbb{M}\times \bar{\X}$ in place of the forcing $\mathbb{M}\times
\mathbb{A}$.  In \cite{su4}, $\mathbb{A}$ is a forcing to add some
Cohen subsets but we only used its chain condition in the proof.  Note further
that the forcing $\hat{\X}$ preserves the chain condition and closure of the
posets that are needed in the proof.

\section{The failure of approachability in the final model} \label{finalmodel}

Let $R$ be $\mathbb{R}$-generic.  We give notation for the generic objects added
by $R$.  Let $\langle x_n \mid n < \omega \rangle$ be the Prikry generic
sequence.  For all $n<\omega$ we let $\lambda_n = x_n \cap \kappa$ and for all
$i \leq \omega+3$, $\lambda_{n,i} = (\alpha \mapsto \alpha_i)(\lambda_n)$ where
$\alpha \mapsto \alpha_i$ is the function defined at the beginning of Section 2.

For $n\geq1$ we have the following generic objects induced by $R$:
\begin{enumerate}
\item $Q_n = Q_n^0 \times Q_n^1$ which is generic for $\Q^0(x_{n-1},x_{n})
\times \Q^1(x_{n-1},x_{n})$.
\item $Q_n^0$ can be written as $P_n^0 * S_n^0$ where $P_n^0$ is generic for
$\Add^V(\lambda_{n-1,\omega+2},\lambda_{n})$ and $S_n^0$ is generic for
$\C^+(\Add^V(\lambda_{n-1,\omega+2},\lambda_{n}),
\lambda_{n-1,\omega+3},\lambda_{n})$ over the extension by $P_n^0$.
\item $Q_n^1$ can be written as $P_n^1 * S_n^1$ where $P_n^1$ is generic for
$\Add^V(\lambda_{n-1,\omega+3},\lambda_{n,1})$ and $S_n^1$ is generic for
$\C^+(\Add^V(\lambda_{n-1,\omega+3},\lambda_{n,1}),\lambda_{n},\lambda_{n,1})$.
\item In a cardinal preserving extension, there are generics
$C_n^0,C_n^1$ which are generic for
$\C(\Add^V(\lambda_{n-1,\omega+2},\lambda_{n}),
\lambda_{n-1,\omega+3},\lambda_{n})$ and 
$\C(\Add^V(\lambda_{n-1,\omega+3},\lambda_{n,1}),\lambda_{n},\lambda_{n,1})$
respectively.
\end{enumerate}

Finally, we let $Q_0$ be the induced generic for
$\Coll(\omega,\lambda_{0,\omega})$.

\begin{lemma} In $V[G*H][R]$, $\aleph_{\omega^2+1} \notin
I[\aleph_{\omega^2+1}]$.  \end{lemma}

In joint work with Sinapova \cite{su3}, we provided a sufficient
condition for the failure of weak square in diagonal Prikry extensions.  It is
straightforward to check that $\mathbb{R}$ is a \emph{diagonal Prikry forcing}
as in Definition 19 and also satisfies the hypotheses of Theorem 26 from that
paper.  Hence we have the failure of weak square in the extension.  To prove the
lemma above, we give a direct argument for the failure of approachability and
note that the technique generalizes to give a metatheorem for the failure
approachability in extensions by diagonal Prikry type forcing.

\begin{proof}  Recall that $\vec{K}$ is $\Q^\omega$-generic over $V[G*H]$.  Note
that in $V[G*H][\vec{K}]$, $\kappa$ is $\kappa_{\omega+1}$-supercompact as
witnessed by $U^*$ the measure on $\mathcal{P}_\kappa(\kappa_{\omega+1})$
derived from $j$ and $\R/I$ is $\kappa_{\omega+1}$-cc where $I$ is the induced
$\Q^\omega/fin$-generic object.  By Remark \ref{downwardsabsolute}, it is enough
to show that approachability fails when we force with $\R/I$ over
$V[G*H]^{\Q^{\omega}}$.  Assume for a contradiction that $\langle \dot{a}_\alpha
\mid \alpha < \kappa_{\omega+1} \rangle$ is a name for a sequence witnessing
approachability.  Let $k:V[G*H][\vec{K}] \to M$ be the ultrapower by $U^*$.  By
the construction of $\R$, we can choose a condition $r \in k(\R)$ which forces
that $\kappa_{\omega+1}=\omega_1$.  We let $\gamma = \sup k``\kappa_{\omega+1}$.

It follows that $r$ forces that $\gamma$ is approachable with respect to
$k(\langle \dot{a}_\alpha \mid \alpha < \kappa_{\omega+1} \rangle)$.  So there
is a $k(\R/I)$-name $\dot{A}$ for a subset of $\gamma$ all of
whose initial segments are enumerated on the sequence $k(\langle \dot{a}_\alpha
\mid \alpha < \kappa_{\omega+1} \rangle)$ before stage $\gamma$.  By standard
arguments we can assume that $\dot{A}$ is forced to be closed.  Since
$\cf(\gamma) = \kappa_{\omega+1}$ and it is forced by $r$ that every club subset
of $\kappa_{\omega+1}$ contains a club from the ground model, there is a club
subset $B$ of $\gamma$ which is forced to be a subset of $\dot{A}$.  We let
$C = \{ \alpha \mid k(\alpha) \in B\}$.  It is straightforward to see that $C$
is $<\kappa$-club in $\kappa_{\omega+1}$.  Let $\eta$ be the
$\kappa_{\omega}$-th element in an increasing enumeration of $C$.

We can assume that there is an index $\bar{\gamma} < \kappa_{\omega+1}$ such
that $r$ forces $\dot{A} \cap k(\eta)$ is enumerated before stage
$k(\bar{\gamma})$ in $k(\langle \dot{a}_\alpha \mid \alpha < \kappa_{\omega+1}
\rangle )$.  Now for every $x \subseteq C \cap \eta$ of ordertype $\omega$,
there is a condition $r_x \in \R/I$ which forces that $x \subseteq
\dot{a}_\alpha$ for some $\alpha < \bar{\gamma}$.  Note that for a given $x$,
$r$ witnesses $k$ applied to this statement.

By the chain condition of $\R/I$, we can find a condition which forces that for
$\kappa_{\omega+1}$ many $x$, $r_x$ is in the generic.  This is impossible,
since we can assume that each $\dot{a}_\alpha$ is forced to have ordertype less
than $\kappa$ and hence $\vert \bigcup_{\alpha < \bar{\gamma}}
\mathcal{P}(\dot{a}_\alpha) \vert \leq \kappa$. \end{proof}

\begin{remark} Note that none of the specific properties of $\R$ are used in the
proof above and hence the assumptions of Theorem 26 of
\cite{su3} are enough to show the failure of approachability.
\end{remark}

Next we take care of the successors of singulars below $\aleph_{\omega^2}$.

\begin{lemma} There is a condition of length $0$ in $\mathbb{R}$ which forces
that for all $n\geq 1$, there is a bad scale of length $\aleph_{\omega\cdot n
+1}$ in some product of regular cardinals. \end{lemma}

It follows that for all $n<\omega$, $\aleph_{\omega\cdot n +1} \notin
I[\aleph_{\omega\cdot n +1}]$.

\begin{proof} Working in $V[G*H]$, fix a scale $\vec{f}$ of length
$\kappa^{+\omega+1}$ in some product of regular cardinals.  By standard
arguments there is a $U_0$-measure one set $A$ such that for all $\delta \in A$,
there are stationarily many bad points for $\vec{f}$ of cofinality
$\delta_{\omega+1}$.  
This is absolute to $M[G^**H^*]$.  It follows that
$\kappa$ is in the set given by $j$ applied to $B = \{ \gamma \mid $ there is a
scale of length $\gamma_{\omega+1}$ such that for all $\delta \in A \cap \gamma$
there are stationarily many bad points of cofinality $\delta_{\omega+1} \}$.  It
follows that $B \in U_0$.

Now for each $i \geq 1$, there is a $U_i$-measure one set $A_i$ of $x$ such that
$\kappa_x \in A \cap B$.  For any $x \prec y$ such that for some $i<i'$ $x \in
A_i$ and $y \in A_{i'}$, we have arranged the following property.  Since
$\kappa_x \in A \cap B \cap \kappa_y$ and by the choice of $A$ and $B$,  there
are stationary many bad points of cofinality $\kappa_{x,\omega+1}$ for some
scale of length $\kappa_{y,\omega+1}$.

The condition required for the lemma is any condition length $0$ whose measure
one sets are contained in the $A_i$.  Work below such a condition and fix $n
\geq 1$.  Let $p$ be a condition of length $n+1$ and $\vec{f}$ be a scale of
length $\kappa_{x_n^{p},\omega+1}$ such that there is a stationary set $S$ of
bad points of cofinality $\kappa_{x_0^p,\omega+1}$.  By Corollary
\ref{boundedsets}, it is enough to show that $\vec{f}$ remains a scale with
stationary set of bad points $S$ in the model $V[G*H][\prod_{i<n}Q_i]$.  The
forcing to add $\prod_{i<n} Q_i$ is small relative to $\kappa_{x_n^p,\omega+1}$
and hence it is easy to see that $\vec{f}$ remains a scale and $S$ remains
stationary.  So it is enough to show that every point in $S$ is still bad for
$\vec{f}$ in the extension.

In this extension $\kappa_{x_0^p,\omega+1}$ becomes $\aleph_1$ via
$\Coll(\omega, \kappa_{x_0^p,\omega})$ and every $\aleph_1$-sequence of
ordinals in $V[G*H][\prod_{i<n}Q_i]$ is in the extension by this collapse.  Now
a standard argument shows that for every $\delta \in S$ and every unbounded
subset $A$ of $\delta$ in the extension there is an unbounded subset of $A$ in
$V[G*H]$.  So if $A$ witnesses that $\delta \in S$ is good in the extension,
then there is an unbounded subset of $A$ witnessing that $\delta$ is good in
$V[G*H]$.  This is impossible, so we must have the every point in $S$ is bad in
$V[G*H][\prod_{i<n}Q_i]$.  This completes the proof. \end{proof}

For the cardinals which are not successors of singulars, we apply the scheme
from the previous section.  Throughout the proofs of the following lemmas, we
omit the straightforward but tedious verification that our scheme from Section
\ref{schematic} applies and that the hypotheses of our preservation lemmas hold.

\begin{lemma}  In $V[G*H][R]$, $\aleph_{\omega^2+2} \notin
I[\aleph_{\omega^2+2}]$ \end{lemma}

\begin{proof} Note that by Remark \ref{downwardsabsolute} it is enough to show
the conclusion in an outer model with the same cardinals up to
$\aleph_{\omega^2+2}$.  So we show that it holds in $V[G*H][\vec{K}][\bar{R}]$
where $\vec{K}$ is generic for $\mathbb{Q}^\omega$ and $\bar{R}$ is generic for
$\bar{\mathbb{R}}$ as defined in the extension by $\vec{K}$.  Since
$\bar{\mathbb{R}}$ is $\kappa_{\omega}$-centered in $V[G*H][\vec{K}]$, by
Theorem \ref{GK} it is enough to show that $\kappa_{\omega+2} \notin
I[\kappa_{\omega+2}]$ in $V[G*H][\vec{K}]$.

We write $H$ as $P_0 \times P_1 * \prod_{i \leq \omega+1} C^+_i$ where $P_1$ is
generic for $\P_1 \times \Add(\kappa_1,\theta^+ \setminus \theta) \simeq
\Add(\kappa_1, \theta^+)$.  Again by Remark \ref{downwardsabsolute}, it is
enough to show $\kappa_{\omega+2} \notin I[\kappa_{\omega+2}]$ in the extension
\[ V[\vec{K}][G][C_{\omega+1}][\prod_{i<\omega}C_i][P_0][P_1 \restrict
[\kappa_{\omega+2},\theta^+)][P_1 \restrict
\kappa_{\omega+2} * C_{\omega}^+].
\]
Note that in $V[\vec{K}][G][C_{\omega+1}]$ we have that $\kappa_{\omega+2}$ is
still Mahlo.  So in this model we apply the scheme from Section \ref{schematic}
and Lemma \ref{formerapproximation} where
\begin{enumerate}
\item $\lambda = \kappa_2$,
\item $\M$ is the forcing to add $P_1 \restrict \kappa_{\omega+2} * C_\omega^+$,
\item $\bar{\X}$ is the forcing to add $P_0 \times C_0 \times P_1 \restrict
[\kappa_{\omega+2},\theta^+]$,
\item $\hat{\X}$ is the forcing to add $\prod_{0<i<\omega}C_i$
\end{enumerate}
\end{proof}

\begin{lemma} In $V[G*H][R]$, $\aleph_{\omega^2+3} \notin
I[\aleph_{\omega^2+3}]$. \end{lemma}

The proof is similar to the proof of the previous lemma with some
changes of the details.

\begin{proof} Again by Theorem \ref{GK} and Claim \ref{qomegapres} it is enough
to show that $\kappa_{\omega+3} \notin I[\kappa_{\omega+3}]$ in
$V[G*H][\vec{K}]$.  By the proof of Claim \ref{qomegapres}, there is a cardinal
preserving outer model of $V[G*H][\vec{K}]$ where we have decomposed $\vec{K}$
as $K^0 \times K^1$ which is generic for the product of
$\kappa_{\omega+3}$-closed forcing and $\kappa_{\omega+3}$-cc forcing both taken
from $V$.  The $\kappa_{\omega+3}$-cc forcing is just
$\Add(\kappa_{\omega+2},\eta)$ for some $\eta$.

As in the previous lemma we pass to an outer model where we have decomposed $H$.
In particular it is enough to prove that $\kappa_{\omega+3} \notin
I[\kappa_{\omega+3}]$ in the model
\[ V[K^0][G][P_0 \times \prod_{i \leq \omega} C_i][K^1][P_1 \restrict
[\kappa_{\omega+3}, \kappa_{\omega+3}^+)][P_1 \restrict \kappa_{\omega+3} *
C^+_{\omega+1}]. \]
Note that in $V[K^0][G]$, $\kappa_{\omega+3}=\theta$ is still Mahlo.  So in this
model we apply the scheme from Section \ref{schematic} and Lemma
\ref{formerapproximation} where
\begin{enumerate}
\item $\lambda = \kappa_{\omega+2}$,
\item $\M$ is the
forcing to add $P_1 \upharpoonright \kappa_{\omega+3} * C_{\omega+1}^+$,
\item $\bar{\X}$ as the forcing to add $P_0\times P_1 \restrict
[\kappa_{\omega+3},\kappa_{\omega+3}^+) \times
\prod_{i \leq \omega}C_i$, and
\item $\hat{\X}$ as the forcing to add $K^1$.
\end{enumerate}
\end{proof}

\begin{lemma} \label{finalargument} In $V[G*H][R]$ for each successor $\tau$ of
a regular cardinal with $\tau \in [\aleph_2, \aleph_{\omega^2})$, $\tau \notin
I[\tau]$.  \end{lemma}

\begin{proof}  There are a few cases based on how close $\tau$ is to the
collapses between the Prikry points.  Some cardinals we must treat individually
and others we can treat uniformly.

First we assume that $\tau = \aleph_2$.  Note that in $V[G*H][R]$,
$\lambda_{0,\omega+1}=\aleph_1$ and $\lambda_{0,\omega+2} = \aleph_2$.  Notice
that any sequence witnessing that $\aleph_2 \in I[\aleph_2]$ is in the extension
$V[G\upharpoonright \lambda_0+1][Q_0 \times Q_1]$.  Recall that $Q_0$ is generic
for $\Coll(\omega,\lambda_{0,\omega})$ and $Q_1$ is generic for
$\lambda_{0,\omega+2}$-closed forcing from $V$.

By Theorem \ref{GK} it is enough to show that $\tau \notin I[\tau]$ in the model
$V[G\upharpoonright \lambda_0+1][Q_1]$.  The proof of this is simpler than the
proof of the next case, so we continue.

Next we assume that $\tau = \lambda_{n,\omega+2}$ where $0<n<\omega$  Any
sequence witnessing that $\tau \in I[\tau]$ in $V[G*H][R]$ is in the extension
$V[G\upharpoonright \lambda_n+1][\prod_{k\leq n} Q_k][Q_{n+1}]$.  As before we
let $P_0 \times P_1$ be generic for $\P_0(\lambda_n) \times (\P_1(\lambda_n)
\times \Add(\lambda_{n,1}, \theta^+_{\lambda_n} \setminus \theta_{\lambda_n}))$
and $\prod_{k \leq \omega+1} C_k^+$ be generic for $\C^+(\lambda_n)$.

By Remark \ref{downwardsabsolute}, it is enough to show that $\tau \notin
I[\tau]$ in the model
\[
V[G \restrict \lambda_n][C_{\omega+1}][Q_{n+1}][\prod_{k \leq n}
Q_k][C_0][P_0][P_1 \restrict [\lambda_{n,\omega+2},
\theta_{\lambda_n}^+)][\prod_{0<i<\omega}C_i][P_1 \restrict
\lambda_{n,\omega+2} * C_\omega^+]
\]
Note that in $V[G \restrict \lambda_n][C_{\omega+1}][Q_{n+1}]$,
$\lambda_{n,\omega+2}$ is still Mahlo.  So in this model we apply the scheme
from Section \ref{schematic} and Lemma \ref{formerapproximation} where
\begin{enumerate}
\item $\lambda$ is $\lambda_{n,2}$,
\item $\M$ is the forcing to add $P \restrict \lambda_{0,\omega+2} *
C_\omega^+$,
\item $\bar{\X}$ is the forcing to add $P_0 \times P_1 \restrict [
\lambda_{n,\omega+2},\theta_{\lambda_n}^+)\times C_0$ and
\item $\hat{\X}$ is the forcing to add $\prod_{0<i<\omega}C_i$.
\end{enumerate}
This completes the argument that $\lambda_{n,\omega+2} \notin
I[\lambda_{n,\omega+2}]$ for $0<n<\omega$.

Suppose that $\tau = \lambda_{n,\omega+3}$ for $n<\omega$.  Any sequence
witnessing $\tau \in I[\tau]$ in $V[G*H][R]$ is in the extension by
$V[G\upharpoonright \lambda_{n+1}][\prod_{k \leq n+1} Q_k]$.  Recall that by
passing to an outer model, we can decompose $Q_{n+1}$ as
$P_{n+1}^0 \times C_{n+1}^0 \times P_{n+1}^1 \times C_{n+1}^1$.  As before we
let $P_0 \times P_1$ be generic for $\P_0(\lambda_n) \times (\P_1(\lambda_n)
\times \Add(\lambda_{n,1},\theta_{\lambda_n}^+ \setminus
\theta_{\lambda_n}))$ and $\prod_{k \leq \omega+1} C_k^+$ be
generic for $\C^+(\lambda_n)$.

By Remark \ref{downwardsabsolute} it is enough to show that $\tau \notin
I[\tau]$ in the model
\[ V[G \restrict \lambda_n][\prod_{k \leq n} Q_k][P_{n+1}^1 \times C_{n+1}^1
\times C_{n+1}^0][P_{n+1}^0][P_0][P_1 \restrict
[\theta_{\lambda_n},\theta_{\lambda_n}^+)][\prod_{k\leq\omega}C_k][P_1 \restrict
\lambda_{n,\omega+3} * C_{\omega+1}^+] \]
We have that $\lambda_{n,\omega+3}$ is Mahlo in the model $ V[G \restrict
\lambda_n][\prod_{k \leq n} Q_k][P_{n+1}^1 \times C_{n+1}^1 \times C_{n+1}^0]$.
So we apply the scheme from Section \ref{schematic} and Lemma
\ref{formerapproximation} in this model where
\begin{enumerate}
\item $\lambda = \lambda_{n,\omega+2}$,
\item $\M$ is the forcing to add $P \restrict
\lambda_{n,\omega+3} *C_{\omega+1}^+$,
\item $\bar{\X}$ is the forcing to add $P_0 \times P_1
\restrict [\theta_{\lambda_n},\theta_{\lambda_n}^+) \times
\prod_{k\leq\omega}C_k$ and
\item $\hat{\X}$ is the forcing to add $P_{n+1}^0$.
\end{enumerate}
This completes the argument that $\lambda_{n,\omega+3} \notin
I[\lambda_{n,\omega+3}]$ for all $n<\omega$.

Next we assume that $\tau = \lambda_n$ for $n \geq 1$.  Any sequence witnessing
that $\tau \in I[\tau]$ in $V[G*H][R]$ is in $V[G \upharpoonright
\lambda_n+1][\prod_{i\leq n} Q_i]$.  As before by passing to an outermodel, we
can decompose $Q_{n}$ as $P_{n}^0 * S_{n}^0 \times P_{n}^1 \times C_{n}^1$.  By
Remark \ref{downwardsabsolute}, it is enough to show that there are no such
sequences in the outer model
\[ V[C_n^1][Y][Z][G \restrict \lambda_{n-1}
+1][\prod_{i \leq n-1}Q_i][P_{n}^1][P_{n}^0* S_{n}^0\]
where $Y$ is generic for $\mathcal{A}(\A \restrict \lambda_n, \A(\lambda_n))$
and $Z$ is generic for $\mathcal{A}(\A) \restrict [\lambda_{n-1}+1,\lambda_n)$.

By Fact \ref{term-iteration}, $\lambda_n$ is still Mahlo in $V[C_n^1][Y][Z]$.
Moreover, the computation of $\Q_n^0 =
\M(\lambda_{n-1,\omega+2},\lambda_{n-1,\omega+3}, \lambda_n)$ is the same in $V$
and $V[C_n^1][Y][Z]$, since the forcing to add $Z$ is closed beyond the first
inaccessible above $\lambda_{n-1,\omega+3}$.

So in this model we apply the scheme from Section \ref{schematic} and Lemma
\ref{formerapproximation} where
\begin{enumerate}
\item $\lambda = \lambda_{n-1,\omega+3}$.
\item $\M$ is the forcing to add $P_n^0*S_n^0$,
\item $\bar{\X}$ is the forcing to add $G \upharpoonright (\lambda_{n-1}+1) \times
\prod_{i\leq n-1}Q_i$
\item $\hat{\X}$ is the forcing to add $P_n^1$ and
\end{enumerate}

This finishes the proof that for all $n \geq 1$, $\lambda_n \notin I[\lambda_n]$
in $V[G*H][R]$.

Next we assume that $\tau = \lambda_{n,1}$ for some $n \geq 1$.  Any sequence
witnessing that $\tau \in I[\tau]$ is in the model $V[G \upharpoonright
\lambda_n+1][\prod_{i \leq n} Q_i]$.  By Remark \ref{downwardsabsolute} it is
enough to show that $\tau \notin I[\tau]$ in the model
\[
V[Y][G \upharpoonright \lambda_n][P_n^0][C_n^0][Y_{P_0}][P_n^1 *S_n^1]
\]
where $Y$ is generic for the poset of $\A\upharpoonright \lambda_n$-names for
elements of $\P_1(\lambda_n) \times \Add(\lambda_{n,1},\theta_{\lambda_n}^+
\setminus \theta_{\lambda_n}) \times \C(\lambda_n)$ and $Y_{P_0}$ is generic for
the poset of $\A\upharpoonright \lambda_n$-names for elements of
$\P_0(\lambda_n)$.  By Fact \ref{cohenterms}, we can take the forcing to add
$Y_{P_0}$ to be $\Add(\lambda_{n,0},\lambda_{n,2})$ as computed in $V$.

Note that $\lambda_{n,1}$ is still Mahlo in $V[Y]$.  So we apply the scheme from
Section \ref{schematic} and Lemma \ref{formerapproximation} in this model where

\begin{enumerate}
\item $\lambda = \lambda_n$,
\item $\M$ is the forcing to add $P_n^1 *S_n^1$,
\item $\bar{\X}$ is the forcing to add $G \upharpoonright \lambda_n \times P_n^0
\times C_n^0$ and
\item $\hat{\X}$ is the forcing to add $Y_{P_0}$.
\end{enumerate}

This finishes the proof that $\lambda_{n,1} \notin I[\lambda_{n,1}]$ for $n \geq
1$.

Next we assume that $\tau = \lambda_{n,2}$ for $n \geq 1$.  Any sequence
witnessing that $\tau \in I[\tau]$ in $V[G*H][R]$ is in $V[G*H][\prod_{k \leq n}
Q_k]$.  By Remark \ref{downwardsabsolute}, it is enough to show that $\tau
\notin I[\tau]$ in the model
\[
V[G\upharpoonright \lambda_n][\prod_{1 \leq i \leq \omega+1}C_i][\prod_{k
\leq n}Q_k][P_1][P_0*C_0^+].
\]

Note that $\lambda_{n,2}$ is still Mahlo in $V[G \upharpoonright
\lambda_n][\prod_{1 \leq i \leq \omega+1} C_i]$.  Hence we apply the scheme from
Section \ref{schematic} and Lemma \ref{formerapproximation} in this model where
\begin{enumerate}
\item $\lambda= \lambda_{n,1}$,
\item $\M$ is the forcing to add $P_0*C_0^+$,
\item $\bar{\X}$ is the forcing to add $\prod_{k \leq n}Q_k$ and
\item $\hat{\X}$ is the forcing to add $P_1$.
\end{enumerate}

Next we assume that $\tau = \lambda_{n,i+1}$ for $0 < n <\omega$ and $2 \leq i
<\omega$.  Any sequence witnessing that $\tau \in I[\tau]$ in $V[G*H][R]$ is in
$V[G \upharpoonright \lambda_{n} +1][\prod_{k\leq n} Q_k]$.  By Remark
\ref{downwardsabsolute}, it is enough to show that $\tau \notin I[\tau]$ in the
model
\[
V[G \upharpoonright \lambda_n][\prod_{k \geq i} C_k][\prod_{k
\leq n}Q_i][\prod_{k<i-1}C_k][P_0][P_1 \upharpoonright [
\lambda_{n,i+1},\theta_{\lambda_n}^+)][P_1\upharpoonright \lambda_{n,i+1}*C_{i-1}^+].
\]

Note that in $V[G \upharpoonright \lambda_n][\prod_{k \geq i} C_i]$,
$\lambda_{n,i+1}$ is still Mahlo.  Hence we apply the scheme from Section
\ref{schematic} and Lemma \ref{formerapproximation} in this model where
\begin{enumerate}
\item $\lambda = \lambda_{n,i}$,
\item $\M$ is the forcing to add $P \restrict \lambda_{n,i+1} *C_{i-1}^+$,
\item $\bar{\X}$ is the forcing to add $\prod_{k\leq n}Q_k \times \prod_{k<
i-1}C_k \times P_0 \times P_1 \restrict [\lambda_{n,i+1},\theta_{\lambda_n}^+)$
and
\item $\hat{\X}$ is the trivial forcing.
\end{enumerate}

This finishes the proof that for $0<n<\omega$ and $2 \leq i < \omega$,
$\lambda_{n,i+1} \notin I[\lambda_{n,i+1}]$ and with it the proof of Lemma
\ref{finalargument} \end{proof}

\bibliographystyle{amsplain}
\bibliography{refs}

\end{document}